\documentclass[12pt]{article}
\usepackage[a4paper,top=2cm,bottom=2cm,left=1.5cm,right=1.5cm]{geometry}
\usepackage{enumerate}
\usepackage{amsmath}
\usepackage{amssymb,latexsym}
\usepackage{amsthm}
\usepackage{color}
\usepackage{graphicx}
\usepackage{hyperref}
\usepackage{amscd}

\title{A lower bound on the minimum weight \\
	of some geometric codes}
\author{Bence Csajb\'ok\thanks{This paper was supported by the János Bolyai Research Scholarship of the Hungarian Academy of Sciences and partially by the ELTE TKP 2021-NKTA-62 funding scheme.}
	, Giovanni Longobardi, Giuseppe Marino, Rocco Trombetti}

\date{}

\newcommand{\cC}{{\mathcal C}}

\newcommand{\cM}{{\mathcal M}}

\newcommand{\cH}{{\mathcal H}}
\newcommand{\cD}{{\mathcal D}}

\newcommand{\cS}{{\mathcal S}}

\newcommand{\F}{{\mathbb F}}

\newtheorem{theorem}{Theorem}[section]

\newtheorem{lemma}[theorem]{Lemma}
\newtheorem{corollary}[theorem]{Corollary}
\newtheorem{definition}[theorem]{Definition}
\newtheorem{proposition}[theorem]{Proposition}
\newtheorem{result}[theorem]{Result}
\newtheorem{example}[theorem]{Example}

\usepackage{todonotes}

\DeclareMathOperator{\PG}{{PG}}
\DeclareMathOperator{\AG}{{AG}}

\begin{document}
	\maketitle
	
	\begin{abstract}
		
		The $p$-ary code associated with the incidence structure of points and $t$-spaces in a projective space $\PG(m,q)$, where $q=p^h$, is the $\F_p$-subspace generated by the incidence vectors of the blocks of this design. The dual of this code consists of all vectors orthogonal to every codeword of the original code. 
		In contrast to the codes derived from point–subspace incidences, the minimum weight of the corresponding dual codes is generally unknown, which makes the problem more challenging.

		
		In 2008 Lavrauw, Storme and Van de Voorde proved the following reduction: the minimum weight of the dual of the code derived from point and $t$-space incidences in $\PG(m,q)$ is the same as the minimum weight of the dual of the code derived from point and line incidences in $\PG(m-t+1,q)$. After a series of works by Delsarte (1970), Assmus and Key (1992), Calkin, Key and De Resmini (1999), the best known lower bound for the case of point-line incidences was established in [B. Bagchi and P. Inamdar:  Projective geometric codes, J.\ Combin.\ Theory Ser.\ A, 99(1) (2002), 128--142].
		
		The problem of determining the minimum weight of these codes admits a natural geometric interpretation in terms of multisets of points in a projective space which meet each line in $0$ modulo $p$ points. 
		In this paper, by adopting this geometrical perspective and exploiting certain polynomial techniques from [S. Ball, A. Blokhuis, A. G\'acs, P. Sziklai, Zs. Weiner: On linear codes whose weights and length have a common divisor, Adv. Math., 211 (2007), 94--104], we prove a substantial improvement of the Bagchi--Inamdar bound in the case where $h>1$ and $m, p >2$.

	\end{abstract}
	
	
	
	\section{Introduction}
	\label{intro}
	
	Let $\Sigma$ be either the projective space $\PG(m,q)$, or the affine space $\AG(m,q)$, where $q=p^h$ for some prime $p$ and integer $h \geq 1$.
	Let $\cD_{\Sigma}(m,q)$ be the design of points and lines of $\Sigma$, i.e. the $2-(v, k, 1)$ design with 
	\[v = \frac{q^{m+1}-1}{q-1},\, k=q+1,\] 
	or \[v = q^m,\, k=q,\] respectively. 
	The $p$-ary code $\cC_{\Sigma}(m,q)$ associated with $\cD_{\Sigma}(m,q)$ is the $\F_p$-subspace generated by the incidence vectors of the blocks of the corresponding design. The \textit{dual} $\cC^\perp$ of a code $\cC$ is the $\F_p$-subspace of vectors orthogonal to all vectors of $\cC$ (under the standard inner product).
	The \textit{weight} $w({\bf v})$ of a vector ${\bf v}\in \F_q^v$ is the number of its non-zero coordinates w.r.t.\ the canonical basis of $\F_q^v$. The \textit{minimum weight} of a code is the minimum of the weights of its non-zero elements. 
	By Assmus and Key \cite[Theorem 5.7.9]{AssmusKey}, Calkin, Key and De Resmini \cite[Proposition 1]{DeResmini} and Delsarte \cite{Delsarte}, the lower bound 
	\begin{equation}
		\label{Del}
		(q+p)q^{m-2}
	\end{equation}
	was known for the minimum weight of both $\cC_{\AG}(m,q)^\perp$ and $\cC_{\PG}(m,q)^\perp$. This bound was improved in 2002 by Bagchi and Inamdar, see \cite[Theorem 3]{BagchiInamdar}, to
	\[2\left (\frac{q^{m}-1}{q-1}\left (1-\frac{1}{p} \right )+\frac{1}{p} \right ).\]

	By \cite[Theorems 10 and 11]{LaStoVan}, the minimum weight codewords of the dual of the code arising from incidences between points and $k$-spaces of $\PG(m,q)$ are the same as the minimum weight codewords of $\cC_{\PG}(m-k+1,q)^\perp$. 
	In \cite[Corollary 1.5]{deboeckvan} De Boeck and Van de Voorde  increased the lower bound on the minimum weight of $\cC_{\PG}(2,p^2)^\perp$, $p \geq 5$ prime, from $2q-2p+2$ to $2q-2p+5$.
	
	If $q$ is even, the lower bound in $\eqref{Del}$ is sharp; see \cite[Corollary 1]{DeResmini} by Calkin, Key and De Resmini.
	In fact, if $q$ is even, then investigating the minimum weight of $\mathcal{C}_{\Sigma}(m, q)^\perp$ is equivalent to finding the minimum size of a non-empty point set $S$ of $\Sigma$ such that each line meets $S$ in an even number of points. These are the so called sets of $\textit{even type}$. The smallest sets of even type were characterised for $q=2$ in \cite[Proposition 3]{BagchiInamdar} and for $q\in \{4,8\}$ in \cite{adriaensen}. 
	
	In the following, we recall the geometrical link between the codewords of $\cC_{\Sigma}(m,q)^\perp$ and certain multisets of $\Sigma$. A \textit{multiset} $\cM$ of $\Sigma$ is a pair $(\cS, \mu)$ where
	\begin{itemize}
		\item [(1)] $\cS$ a non empty set of points of $\Sigma$, and 
		\item [(2)] $\mu \colon \cS \rightarrow \mathbb{Z}^+$ a function that assigns a positive multiplicity to each element of $\cS$.
	\end{itemize}
	If $T$ is a subset of $\Sigma$,  then $|T\cap \cM| = \sum_{x\in T\cap \cS} \mu(x)$. In particular, $|\cM|=\sum_{x\in \cS}\mu(x)$. 
	
	Let $\mathcal{P}=\{x_1,x_2,\ldots,x_v\}$ be the point set of $\Sigma$, and consider $\cM=(\cS,\mu)$, a multiset of $\Sigma$. The vector ${\bf s}=(s_1,\ldots,s_v)$, such that
	\begin{equation}
		s_i=
		\begin{cases}
			\mu(x_i) & \textnormal{if} \,\, x_i \in \cS, \\
			0 & \textnormal{otherwise},
		\end{cases}
		\quad i=1,2,\ldots v,
	\end{equation}
	is called the \textit{characteristic vector} of $\cM$. Let $\sigma({\bf s})$  be the sum of the coordinates of the characteristic vector ${\bf s}$, then $\sigma({ \bf s} ) = \vert \cM \vert$ and if $\mu(x)=1$ for all $x \in \cS$, then $\cM$ is an ordinary set. Moreover, for an ordinary set, the weight $w({\bf s})$ of its characteristic vector ${\bf s}$  is the same as the sum $\sigma({\bf s})$ of the coordinates of ${\bf s}$.\\
	Let $\cS$ denote a non-empty point set of $\PG(2,q)$ meeting  each line of the plane in $0 \bmod r$ points. It follows from easy counting arguments that $r$ has to be a power of $p$ and hence lines meet $\cS$ in $0 \bmod p$ points. It follows that the characteristic vector ${\bf s}$ of $\cS$ is an element of $\cC_{\PG}(2,q)^\perp$. In \cite[Theorem 2.1]{5nomi},  Ball \textit{et al.} proved a lower bound on the size of such point sets. In fact, they proved something more general: a lower bound for the size of non-empty point sets meeting each line in $0 \bmod p^s$ points. What we will cite here is their result in the $s=1$ case, in terms of characteristic vectors. 
	
	\begin{result}\label{result}
		Suppose ${\bf s}\in \cC_{\PG}(2,q)^\perp$ and $q>p$. If ${\bf s}$ has only coordinates $0$ and $1$, then \begin{equation} w ({\bf s}) \geq (p-1)(q+p).
		\end{equation}
	\end{result}
	
	This bound is sharp when $q=p^2$, or when $p=2$, see \cite[Example 4.4]{5nomi}. Result \ref{result} was obtained by first proving it for ${\bf s} \in \cC_{\AG}(2,q)^\perp$. If $\cS$ is an ordinary point set of $\PG(2,q)$ of size less than $p(q+1)$ and meeting each line in $0 \bmod p$ points, then there always exists a line disjoint from $\cS$. The existence of such a line allows to transfer the problem from $\PG(2,q)$ to $\AG(2,q)$. We point out that for a multiset $\cM$ of $\PG(m,q)$, $m \geq 3$, we cannot see how a similar restriction on $|\cM|$ could ensure the existence of a hyperplane disjoint from $\cM$.
	
	From now on, a multiset of $\Sigma$ intersecting each line in $0 \bmod p$ points will be called a $(0 \bmod p)$-\textit{multiset}. Note that $\cM$ is a $(0 \bmod p)$-multiset if and only if its characteristic vector ${\bf s}=(s_1,s_2,\ldots,s_v)$ belongs to $\cC_{\Sigma}^\perp(m,q)$.
	In Section \ref{sec: main-result}, following the main ideas from \cite[Theorem 2.1]{5nomi} 
	we prove the following result.
	
	\begin{theorem}
		\label{codebound}
		Suppose ${\bf s}\in \cC_{\AG}(m,q)^\perp$ and $q> p>2$ with $m\geq 2$. If ${\bf s}$ has a coordinate equal to $1$, then
		\begin{equation}\label{affine_bound}
			\sigma({\bf s}) \geq    (p-1)q^{m-1}+pq^{m-2}.
		\end{equation}
	\end{theorem}
	
	Clearly, if no coordinates in ${\bf s} \in \cC_{\AG}(m,q)^\perp$ equals $1$ and ${\bf s}\neq {\bf 0}$ then by \eqref{Del} it follows that
	\begin{equation}\label{no-1}
		\sigma({\bf s}) \geq 2q^{m-1}+2p q^{m-2}.
	\end{equation}

	In Section \ref{sec_examples}, we show that the lower bound stated in (\ref{affine_bound}) is sharp for any $q>p$ and $m \geq 2$, by showing a family of examples attaining it. When $p=2$, then our construction is the same as the recursive one presented by Calkin \textit{et al.} in \cite[Section 3]{DeResmini} starting from a translation hyperoval of $\PG(2,2^h)$. Finally, using Theorem \ref{codebound}, in Section \ref{sec:lowerbound} we prove a significant improvement of the Bagchi--Inamdar bound in the case  $h>1$ and $m, p > 2$, obtaining:
	\[2\left(q^{m-1}\left(1-\frac1p\right)+q^{m-2}\right).\]
	We conclude the article with the investigation of $\sigma({\bf s})$, with ${\bf s} \in \cC_{\PG}(m,q)^\perp$, in characteristics $3$ and $5$.

	
	

\section{\texorpdfstring{On the size of an affine $(0 \bmod p)$-multiset}{On the size of an affine (0 mod p)-multiset}} \label{sec: main-result}

Let $\cM=(\cS,\mu)$ be a multiset of $\Sigma$. Since we are interested in a lower bound on the size  of multisets meeting lines in $0 \bmod p$ points, from now on we assume that the multiplicities of the points of $\cS$ are between $1$ and $p-1$. In fact, for each point $z \in \cS$, put $\mu^*(z)= \mu(z) \bmod p$ (the residue of $\mu(z)$ modulo $p$), and let $\cS^*$ denote the points of $\cS$ with multiplicity not divisible by $p$. Then $\mathcal{M}^*=(\cS^*,\mu^*)$ is a multiset of size (possibly) smaller than $\cM$ and intersecting each line of $\Sigma$ in  $0 \bmod p$ points. Let us suppose that there exists at least one point $x \in \cS$ such that $\mu(x)=1$, then by counting the lines of $\Sigma$ through $x$, we have 
\begin{equation}\label{eq:trivial}
	\vert \mathcal{M} \vert \geq 1+ (p-1) \frac{q^m-1}{q-1} .
\end{equation}

In terms of multisets, Theorem \ref{codebound} from Section \ref{intro} can be formulated as follows. 


\begin{theorem}
	\label{th:lowerbound}
	Let $\cM=(\cS,\mu)$ be a non-empty $(0 \bmod p)$-type multiset of points in $\AG(m,q)$, $m\geq 2$, $q=p^h$ with $p>2$ and $h>1$. If there is at least one point $x\in \cS$ with $\mu(x)=1$, then 
	\begin{equation}
		\label{lowerbound}
		|\cM|\geq (p-1)(q^{m-1}+q^{m-2})+q^{m-2}. 
	\end{equation}
\end{theorem}

\begin{proof}
	Firstly, note that for $m=2$, the lower bound in the statement coincides with that in \eqref{eq:trivial}. Hence, let us suppose $m >3$. 
	By counting points of $\cM$ on lines incident with a fixed point not in $\cM$, we obtain $p \mid |\cM|$. 
	We will use the same type of representation for $\AG(m,q)$ as the one used in \cite[Section 2]{5nomi} for $\AG(2,q)$. 
	In this way we can identify $\cM$ with a multiset of $\F_{q^m}$. Define the polynomial
	\[
	R(X,Y)=\prod_{b\in \cM}(X+(Y-b)^{q-1})=\sum_{j=0}^{|\cM|}\sigma_j(Y)X^{|\cM|-j}.
	\] 
	For $j<0$ and $j>|\cM|$ we define $\sigma_{j}(Y)$ to be the zero polynomial.
	
	For $y\in \cM$ with $\mu(y)=t$, it holds that
	\begin{equation}
		\label{eq0}
		R(X,y)=X^t\left (X^{\frac{q^{m}-1}{q-1}}+(-1)^{m-1} 
		\right )^{p-t}g_y(X)^p,
	\end{equation}
	where the leading coefficient of $g_y(X)$ is $1$. 
	Whereas, if $y\notin \cM$, then
	\begin{equation}
		\label{eq00}
		R(X,y)=h_y(X)^p,
	\end{equation}
	for some $h_y(X)\in \F_{q^m}[X]$. 
	It means that for each $y\in \F_{q^m}$, if $0<j< \frac{q^{m}-q}{q-1}$ and $p \nmid j$, then $\sigma_j(y)=0$. Then for these $j$'s the polynomial $Y^{q^m}-Y$ divides $\sigma_j(Y)$, which is of degree at most $j(q-1)<q^m$, and hence $\sigma_j(Y)$ is the zero polynomial when $0<j< \frac{q^{m}-q}{q-1}$ and $p \nmid j$. Then we may write
	
	\begin{equation}
		\label{R(X,Y)}
		\begin{split}
			R(X,Y)=&X^{|\cM|}+\sigma_p(Y)X^{|\cM|-p}+\sigma_{2p}(Y)X^{|\cM|-2p}+
			\cdots + \sigma_{q^{m-1}+\ldots+q}(Y)X^{|\cM|-(q^{m-1}+\ldots+q)}+\\
			&\sum^{\vert \cM \vert}_{j=\frac{q^{m}-1}{q-1}}\sigma_{j}(Y)X^{\vert \cM \vert -j}.
		\end{split}
	\end{equation}
	The partial derivative of $R(X,Y)$ with respect to $Y$, evaluated in $y \in \F_{q^m}$, is 
	\[\frac{\partial R}{\partial Y}(X,y)=
	\left(\sum_{b\in \cM} \frac{-(y-b)^{q-2}}{X+(y-b)^{q-1}} \right)R(X,y).\]
	Since the above denominators are all divisors of $X^{\frac{q^{m}-1}{q-1}}+(-1)^{m-1}$, we obtain
	\begin{equation}
		\label{eq}
		R(X,y)G_y(X)=(X^{\frac{q^{m}-1}{q-1}}+(-1)^{m-1})\frac{\partial R}{\partial Y}(X,y)
	\end{equation}
	for some polynomial $G_y(X)\in \F_{q^m}[X]$,  depending on $y$, whose degree is at most $\frac{q^{m}-1}{q-1}-p$.
	
	Next, we prove some properties of the polynomial $G_y(X)$.
	
	Note that at the right-hand side of \eqref{eq} the highest degree, which is not $1 \bmod p$, is at most $|\cM|$.
	Assume $y \notin \cM$, if there was a non-constant term in $G_y(X)$ of degree not $1 \bmod p$ then there would be a term on the left-hand side of \eqref{eq} of degree larger than $|\cM|$ and of degree not $1 \bmod p$ (cf. \eqref{eq00}). This would be a contradiction, which proves that
	\[G_y(X)=c_y+X H_y(X)^p\]
	for some polynomial $H_y(X)\in \F_{q^m}[X]$. To find the constant $c_y \in \F_{q^m}$,  we compare the coefficients of $X^{|\cM|}$ with \eqref{eq}.
	It follows that \[c_y=\sigma_{\frac{q^{m}-1}{q-1}}'(y).\]
	
	\noindent Next, we prove some properties of the $\sigma_j$'s. For each $i$ such that $|\cM|-i$ is not congruent to $0$ or $1$ modulo $p$ (i.e., $i$ is not congruent to $0$ or $-1$ modulo $p$) and $y\notin \cM$, on the left-hand side of \eqref{eq} the coefficients of terms of degree $|\cM|-i$ are equal to zero. 
	
	For each $y \in \F_{q^m}$, on the right-hand side of \eqref{eq} the coefficient of $X^{|\cM|-i}$ is 
	\begin{equation}\label{sigmas'}
		(-1)^{m-1}\sigma'_{i}(y)+\sigma'_{i+(q^m-1)/(q-1)}(y).
	\end{equation}
	We can put
	\[|\cM|=(p-1)(q^{m-1}+q^{m-2})+kp\geq 1+(p-1)\frac{q^m-1}{q-1}.\] 
	Clearly, $kp>0$. To obtain a contradiction, from now on we assume $kp< q^{m-2}$.
	Then
	\begin{equation}\label{sizeM}
		(p-1)\frac{q^{m}-1}{q-1}+kp > |\cM|,
	\end{equation} 
	and hence $\sigma_{(p-1)\frac{q^m-1}{q-1}+kp}(Y)$ is the zero polynomial by definition. Since $(p-2)\frac{q^m-1}{q-1}+kp$ is not congruent to $0$ or $-1$ modulo $p$, on the left-hand side of \eqref{eq} the coefficient of $X^{|\cM|-((p-2)\frac{q^m-1}{q-1}+kp)}$ is equal to $0$ and hence by \eqref{sigmas'}:
	
	\[\sigma'_{(p-2)\frac{q^m-1}{q-1}+kp}(y)=(-1)^m\sigma'_{(p-1)\frac{q^{m}-1}{q-1}+kp}(y).\]
	Here, the right-hand side -- and hence also the left-hand side -- are equal to zero. Iterating the same argument, we obtain $\sigma'_{(p-\ell)\frac{q^m-1}{q-1}+kp}(y)=0$, 
	for any $\ell \in \{1,\ldots,p-1\}$, in particular 
	\[\sigma'_{\frac{q^m-1}{q-1}+kp}(y)=0.\]
	It follows that $(-1)^{m-1}\sigma'_{kp}(y)$ is the coefficient of $X^{|\cM|-kp}$ on the right-hand side of $\eqref{eq}$. At the left-hand side, this coefficient is $\sigma_{kp}(y)c_y=\sigma_{kp}(y)\sigma_{(q^{m}-1)/(q-1)}'(y)$ and hence 
	\begin{equation}
		\label{kell}
		\sigma_{kp}(y)\sigma_{\frac{q^m-1}{q-1}}'(y)=(-1)^{m-1}\sigma'_{kp}(y)
	\end{equation}
	for each $y\notin \cM$. Now we examine $\sigma_{\frac{q^{m}-1}{q-1}}(Y)$. If $y\notin \cM$ then $\sigma_{\frac{q^{m}-1}{q-1}}(y)=0$ (cf. \eqref{eq00}) and hence the degree of $\sigma_{\frac{q^{m}-1}{q-1}}(Y)$ is at least $q^m-|\cM|$, or it is the zero polynomial. 
	If $y\in \cM$ with multiplicity $1\leq t \leq p-1$, then $\sigma_{\frac{q^{m}-1}{q-1}}(y)=(p-t)(-1)^{m-1}=(-1)^mt$ (cf.\ \eqref{eq0}), in particular $\sigma_{\frac{q^{m}-1}{q-1}}(Y)$ is not the zero polynomial.
	
	Put $f(Y)=\prod_{y\in \cM}(Y-y)$. Then  $f(Y)\sigma_{\frac{q^m-1}{q-1}}(Y)$ is zero for each $y\in \F_{q^m}$,  hence, we may put it in the following form:
	\begin{equation}
		\label{deriv}
		f(Y)\sigma_{\frac{q^m-1}{q-1}}(Y)=(Y^{q^m}-Y)T(Y).
	\end{equation}
	Moreover, the degree of $f(Y)\sigma_{\frac{q^m-1}{q-1}}(Y)$ is at most $| \cM| + q^m-1$. This implies that $T(Y)$ has degree at most $|\cM|-1$. For each $y\in \cM$, if $\mu(y)=t$, then, after the derivation of \eqref{deriv}
	\[(-1)^{m-1}t f'(y)=T(y).\]
	If $t=1$, then $(-1)^{m-1}f'(y)=T(y)$, while if $t >1$, $y$ is a multiple root of $f(Y)$ and hence $f'(y)=0$, thus $T(y)=(-1)^{m-1}tf'(y)=0$. Therefore,
	\[(-1)^{m-1}f'(y)=T(y)\]
	holds for each $y \in \cM$ and since $|\cM|$ is larger than the degree of $((-1)^{m-1}f'-T)(Y)$, it follows that $(-1)^{m-1}f'(Y)=T(Y)$.
	
	Again, after derivating \eqref{deriv}, substituting $y\notin \cM$ and applying $\sigma_{\frac{q^m-1}{q-1}}(y)=0$, we obtain
	\[f(y)\sigma'_{\frac{q^m-1}{q-1}}(y)=-T(y)=(-1)^mf'(y).\]
	Then, combined with \eqref{kell}, we obtain
	\[\sigma_{kp}(y)f'(y)=-\sigma'_{kp}(y)f(y)\]
	and hence 
	$(f\sigma_{kp})'(y)=0$, for each $y\notin \cM$.
	Since the degree of $(f\sigma_{kp})(Y)$ is at most $kp(q-1)+|\cM|$, which is divisible by $p$, the polynomial $(f\sigma_{kp})'$ has the degree at most $kp(q-1)+|\cM|-2$. In order for $(f\sigma_{kp})'$ to be the zero polynomial, we need this degree to be less than $q^m-|\cM|$. Equivalently, we need
	\[kp(q-1)+2|\cM|-2<q^m,\]
	which certainly holds when
	\[kp(q-1)+2\left ((p-1)\frac{q^{m}-1}{q-1}+kp-1\right)-2 <   q^m.\]
	This in turn certainly holds when $kp \leq q^{m-2}(q-2p)$. In summary, if $|\cM|<(p-1)\frac{q^m-1}{q-1}+q^{m-2}(q-2p)$, then $(f\sigma_{kp})'$ is the zero polynomial. Since the lower bound in the statement for $|\cM|$ is smaller than this number, we may indeed assume for the rest of the proof that $(f\sigma_{kp})'$ is the zero polynomial. 
	This means that
	$(f\sigma_{kp})(Y)$ is a  $p$-th power of a polynomial in $\F_{q^m}[Y]$, and hence $f(Y)$ divides $\sigma_{kp}(Y)^{p-1}$. 
	Indeed, $f$ is product of linear factors, each of them with multiplicity at most $p-1$. Assume $(Y-c) \mid f(Y)$ for some $c$. Then $c$ is at least $p$-fold root of $f\sigma_{kp}$, so it is also a root of $\sigma_{kp}$. Hence, $(Y-c)^{p-1} \mid \sigma_{kp}(Y)^{p-1}$. Since this holds for each root $c$ of $f$,  $f(Y) \mid \sigma_{kp}(Y)^{p-1}$.

	Moreover, by definition, $f$ is a product of linear factors over $\F_{q^m}$, each of them with multiplicity at most $p-1$. It follows that either $\deg f \leq  (p-1)\deg\sigma_{kp}$, 
	and hence $|\cM|\leq kp(p-1)(q-1)$, or $\sigma_{kp}$ is the zero polynomial.
	
	Recall $kp<q^{m-2}$, and hence
	\[|\cM|=(p-1)(q^{m-1}+q^{m-2})+kp>kp(p-1)(q-1),\] 
	i.e. $\deg f>(p-1)\deg \sigma_{kp}$ and hence $\sigma_{kp}$ is the zero polynomial.
	Note that the coefficient of the term
	$X^{(p-1)(q^{m-1}+q^{m-2})}Y^{kp(q-1)}$ in $R(X,Y)$ is
	$\binom{|\cM|}{kp}$. According to Lucas' theorem, recall 
	$kp<q^{m-2}$,  $\binom{|\cM|}{kp}=1$ and hence it is a non-zero term which appears also in $\sigma_{kp}$, a contradiction. 
	
	Since the assumption $kp<q^{m-2}$ led to a contradiction, the assertion follows.
\end{proof}

Theorem \ref{th:lowerbound} assumes $p>2$, while the analogous question for $p=2$ asks for the minimum size of a non-empty (ordinary) point set $\cS$ in $\AG(m,q)$ such that lines meet $\cS$ in $0 \bmod p$ points. The sharp lower bound for the size of such point sets is given in \eqref{Del}, and it coincides with the value that \eqref{lowerbound} yields when $p=2$.


\medskip
\noindent We conclude this section by noting that if $\mathbf{s}\in \cC_{\mathrm{AG}}(m,q)^\perp$, then we can extend $\mathbf{s}$ with $(q^m-1)/(q-1)$ zeroes (corresponding to points of the hyperplane at infinity) to obtain a codeword $\bar{\mathbf{s}}$ in $\cC_{\mathrm{PG}}(m,q)^\perp$. Clearly, $w(\bar{\mathbf{s}})=w(\mathbf{s})$ and $\sigma(\bar{\mathbf{s}})=\sigma(\mathbf{s})$; hence, the minimum weight of $\cC_{\mathrm{AG}}(m,q)^\perp$ is at least the minimum weight of $\cC_{\mathrm{PG}}(m,q)^\perp$.

\section{Examples attaining the lower bound}\label{sec_examples}

This subsection is devoted to the construction of a family of multisets in $\AG(m,q)$ that contains at least one point of multiplicity $1$, such that every affine line intersects the multiset in a number of points divisible by $p$, and whose size attains the bound given in \eqref{lowerbound}.

\medskip

\noindent Let $V$ denote the $(m+1)$-dimensional $\F_q$-vector space underlying $\Sigma=\PG(m,q)$. If $q=p^h$, then we can view $V$ as an $h(m+1)$-dimensional $\F_p$-vector space. If $W$ is an $r$-dimensional $\F_p$-subspace of $V$, then the set of points $L_W$ of $\Sigma$ defined by non zero vectors of $W$ is called an $\F_p$-\textit{linear set} of \textit{rank} $r$. It is straightforward to see that  $\vert L_W \vert \leq \frac{p^r-1}{p-1}$. Lines of $\Sigma$ correspond to sets of points defined by non-zero vectors of $2$-dimensional $\F_q$-subspaces of $V$, which can be viewed as $2h$-dimensional $\F_p$-subspaces. 
By Grassmann's Identity, if $U$ is an $(h(m-1)+1)$-dimensional $\F_p$-subspace of $V$, then it meets every $2$-dimensional $\F_q$-subspace non-trivially. If $T$ is any $2$-dimensional $\F_q$-subspace and $\ell$ is the line of $\Sigma$ defined by non zero vectors of $T$, then it follows that $\ell \cap L_U\neq \emptyset$. Moreover, it can be seen that $|\ell \cap L_U |\equiv 1 \pmod p$, see \cite{Bonoli-Polverino} and \cite[Proposition 2.2]{linearset}.

The subsequent results are crucial for achieving our goal.

\begin{proposition}
	\label{prop1}
	Let $L_U$ and $L_W$ be any two $\F_p$-linear sets of rank $h(m-1)+1$ in $\PG(m,q)$, $q=p^h$, $h>1$. Put $\cS=L_U \triangle L_W$ (the symmetric difference of $L_U$ and $L_W$) and define $\mu$ as follows:
	\begin{equation}\label{value}
		\mu(x)=
		\begin{cases}
			t & \textnormal{if} \, x \in L_U \setminus L_W\\ 
			p-t & \textnormal{if} \, x\in L_W\setminus L_U,
		\end{cases}
	\end{equation}
	where $1 \leq t \leq p-1$ is fixed. Then $\cM=(\cS,\mu)$ is a multiset meeting each line of $\PG(m,q)$ in $0 \bmod p$ points.
\end{proposition}
\begin{proof}
	Let $\ell$ be any line of $\PG(m,q)$ and put $a=|\ell \cap (L_U\setminus L_W)|$, $b=|\ell \cap (L_U \cap L_W)|$ and $c=|\ell \cap (L_W\setminus L_U)|$. From above it holds that $a+b\equiv 1 \equiv b+c  \pmod p$. Moreover, we also have the following:
	\begin{equation*}
		|\cM \cap \ell|= ta+(p-t)c\equiv 0 \pmod p.
	\end{equation*}
\end{proof}

An $\F_p$-subspace $U$ of $H\cong \F_q^m$ is said to be an $(u,v)_p$-\textit{evasive subspace} if the $\F_q$-subspace $ \langle U \rangle_{\F_{q}}$ has dimension at least $u$ over $\F_q$ and the $u$‐dimensional $\F_q$-subspaces of $H$ meet $U$ in  $\F_p$‐subspaces of dimension at most $v$. For more details see \cite{evasive}. We will need the following result on $(m-1,(m-2)h)_p$-evasive subspaces: 

\begin{result}[{\cite[Special case of Corollary 4.4]{evasive}}]
	\label{eva}
	Let $U$ be an $\F_p$-subspace of $H\cong \F_q^m$, $q=p^h$, and assume that $U$ meets every $(m-1)$-dimensional $\F_q$-subspace of $H$ in an $\F_p$-subspaces of dimension at most $(m-2)h$. Then $\dim_p (U) \leq (m-1)h-1$.
\end{result}

\begin{proposition}
	\label{prop2}
	Let $f\colon \F_q^{m-1}\rightarrow \F_q$ be an $\F_p$-multilinear map, $m \geq 2$. In $V= \F_q^{m+1}$ put
	\begin{equation} \label{U}
		U=\{ (x_0,x_1,\ldots,x_{m-2},f(x_0,x_1,\ldots,x_{m-2}),y) \colon x_0,x_1,\ldots,x_{m-2} \in \F_q,\, y\in \F_p \},
	\end{equation}
	\begin{equation}\label{W} 
		W=\{(x_0,x_1,\ldots,x_{m-2},y,0)\colon x_0,x_1,\ldots,x_{m-2}\in \F_q, \, y\in \F_p\}.     
	\end{equation}
	
	Define $\cM=(L_U \triangle L_W,\mu)$, a multiset of $\PG(m,q)$, with $\mu$ as in Proposition \ref{prop1} with $t=1$.
	\begin{enumerate}[\rm(1)]
		\item There is a hyperplane disjoint from $\cM$.
		\item $\cM$ is of size
		\[(p-1)(q^{m-1}+q^{m-2})+q^{m-2}\]
		if and only if 
		\[|L_U \cap L_W|=q^{m-2}\frac{q-1}{p-1}+\frac{q^{m-2}-1}{q-1}. \]
		
	\end{enumerate}
\end{proposition}
\begin{proof}
	$(1)$ First, observe that $L_W=\PG(H,\F_q)$ where $H$ is an $m$-dimensional $\F_q$-vector space. Hence,  $L_U \cap L_W=L_{U\cap H}$ is an $\F_p$-linear set of rank $(m-1)h$. By Grassmann's Identity, $L_{U\cap H}$ meets hyperplanes of $L_W$ in linear sets of rank at least $(m-2)h$. By Result \ref{eva}, it follows that $L_{U \cap H}$ cannot meet every hyperplane of $L_W$ in an $\F_p$-linear set of rank exactly $(m-2)h$. 
	Hence, $L_{U \cap H}$ meets at least a hyperplane, say $\Pi$ of $L_W$, in an $\F_p$-linear set of rank at least $(m-2)h+1$. Then, again by Grassmann's Identity, $\Pi$ must be fully contained in $L_{U\cap H}$. 
	To prove $(1)$, it is enough to find a hyperplane of $\PG(m,q)$ through $\Pi$ which does not meet $\cM$. Assume for the contrary, 
	that each hyperplane $\cH$ of $\PG(m,q)$ through $\Pi$ and distinct from $L_W \colon X_m=0$ meets $\cM$. Since $\cM \cap \cH$ is a multiset of $\cH \setminus \Pi \cong \AG(m-1,q)$ meeting every line in $0 \bmod p $ points and having only points of multiplicity $1$, by applying Theorem \ref{th:lowerbound}, we have that 
	$$q^{m-1}= \vert L_U \setminus L_W \vert = \sum_{\substack{\Pi \subset \cH \\ \cH \neq L_W}} |\cM \cap \cH|  \geq q ((p-1)(q^{m-2}+q^{m-3})+q^{m-3}),$$
	a contradiction.
	
	To prove $(2)$, it is enough to note that, by construction, 
	$$|\cM|= (p-1)\vert L_W \setminus (L_{U} \cap L_{W}) \vert + \vert L_U \setminus L_W \vert $$
	and  $|L_U \setminus L_W|= q^{m-1}$.
\end{proof}

By Proposition \ref{prop2}, we can get an upper bound on the size of an $\F_p$-linear set of rank $hm$ in $\PG(m,q)$, $m \geq 2$, $q=p^h$, and $h >1$. Indeed, the following holds.

\begin{theorem} \label{direction}
	Let $L$ be a $\F_p$-linear set of rank $hm$ in $\PG(m,q)$, $q=p^h$. Then,
	\[\vert L \vert \leq q^{m-1}\frac{q-1}{p-1}+\frac{q^{m-1}-1 }{q-1}.\]
\end{theorem}

\begin{proof}
	After a suitable collineation, for any $\F_p$-linear set $L$ of rank $hm$ in $\PG(m,q)$ we may assume that $(0:0:\ldots :0:1)\notin L$ and hence $L$ can be written as $L=L_U$,  where
	\[U=\{(x_0,x_1,\ldots,x_{m-1}, f(x_0,\ldots,x_{m-1})) : x_i \in \F_{q^m}\}\] 
	$U$ is an $\F_p$-subspace of dimension $hm$, and hence $f\colon \F_{q}^{m} \rightarrow \F_q$ is additive. 
	Let us consider  the $\F_p$-linear sets $L_{\overline{U}}$ and $L_{\overline{W}}$ of $\PG(m+1,q)$, 
	where
	\[
	\overline{U}=\{(x_0,x_1,\ldots,x_{m-1},f(x_0,\ldots,x_{m-1}), y): x_0,\ldots,x_{m-1} \in \F_q, y \in \F_p\},
	\]
	\[
	\overline{W}=\{(x_0,x_1,\ldots,x_{m-1},y,0): x_0,\ldots,x_{m-1} \in \F_q, y \in \F_p\}.
	\]
	Then, by Proposition \ref{prop1} with $t=1$, the multiset $\bar{\cM}=(L_{\overline{U}} \triangle L_{\overline{W}}, \mu)$ with $\mu$ as in \eqref{value}
	meets any line of $\PG(m+1,q)$ in $0 \bmod p$ points.
	By Proposition \ref{prop2} $(2)$,
		\[
		|L_U| =|L_{\overline{U}} \cap L_{\overline{W}}|  \leq q^{m-1}\frac{q-1}{p-1}+\frac{q^{m-1}-1 }{q-1},
		\]
		
		hence the result follows.
	\end{proof}
	
	The trivial upper bound for the size of an $\F_p$-linear set of rank $hm$ in $\PG(m,q)$, $q=p^h$, is $\frac{p^{hm}-1}{p-1}=\frac{q^m-1}{p-1}$.  
	It is easy to see that our bound is stronger, indeed $\frac{q^{m}-1}{p-1} >  q^{m-1}\frac{q-1}{p-1}+\frac{q^{m-1}-1 }{q-1}$ follows from
	\[     (q-1)(q^m-1) - (q^{m-1}(q-1)^2+(q^{m-1}-1)(p-1)) = (q-p)(q^{m-1}-1) >0.
	\] 
	
	\bigskip 
	In the remainder, we apply Propositions \ref{prop1} and \ref{prop2} to show that in $\AG(m,q)$ there exist $(0 \bmod p)$-multisets having at least one point with multiplicity $1,$ with size $(p-1)(q^{m-1}+q^{m-2})+q^{m-2}$.
	
	\medskip
	\noindent To this aim we first recall that an $\F_p$-linearized polynomial  $f(x)=\sum_{i=0}^{r-1}a_i x^{p^i}\in \F_q[x]$ is called \textit{scattered} if 
	$$ \left \vert \left \{\frac{f(x)}{x} \colon x \in \F_q^* \right \} \right \vert= \frac{q-1}{p-1}, $$
	for a survey on scattered polynomials, see \cite{survey}.
	
	Now, assume $L_U \subset \PG(m,q)$ to be an $\F_p$-linear set whose underlying $\F_p$-vector space is
	\[ U=\{ (x_0,x_1,\ldots,x_{m-2},f(x_0),y) \colon x_0,x_1,\ldots,x_{m-2} \in \F_q,\, y\in \F_p \},\] where $f$ is any scattered $\F_p$-linearized polynomial. Also, consider the $\F_p$-subspace $W$ of $\F_{q}^{m+1}$ as in \eqref{W}. It is easy to see that $L_U \cap L_W$ is an $\F_p$-linear set contained in the hyperplane of $\PG(m,q)$ with equation $X_m=0$ and it is a cone with basis a maximum scattered $\F_p$-linear set on the line with equations $X_1=X_2=\ldots=X_{m-2}=X_m=0$ and vertex an $(m-3)$-dimensional subspace with equations $X_0=X_{m-1}=X_m=0$. Then,
	$$ \vert L_U \cap L_W \vert =q^{m-2}\frac{q-1}{p-1}+ \frac{q^{m-2}-1}{q-1},$$
	and  $\cM= (L_U \triangle L_W, \mu)$ as defined in \eqref{value}  is a multiset of $\AG(m,q)$ whose size attains the bound in \eqref{lowerbound}. The fact that this construction can be embedded in $\AG(m,q)$ follows from Proposition \ref{prop2} $(1)$.
	
	We note that the characteristic vector of the construction above has the same weight as the one associated with the point set constructed by Lavrauw, Storme and Van de Voorde in \cite[Theorem 4.14] {linearcodes}.  In particular, if $f(X)=X^p$, then our construction has the same support as the one in \cite[Theorem 4.14] {linearcodes}. Moreover, if $p=2$ the example falls into the family exhibited by Calkin, Key and De Resmini in \cite{DeResmini}. Indeed, in such a case $L_U$ is equivalent to a cone with base the $\F_2$-linear set $\{\langle (x,x^{2^s})\rangle_{\F_{2^h}}\,:\, x \in \F_{2^h}^*\} \subseteq \PG(1,2^h)$, with $\gcd(s,h)=1,$ which is the set of directions determined by the points of a translation hyperoval of $\PG(2,2^h)$ (see \cite{payne_complete_1971}). In such a case, our construction provides a set which is equivalent to the one constructed by induction in \cite[Note $1$]{DeResmini}.

	\section{\texorpdfstring{A lower bound on the minimum weight of $\cC_{\PG}(m,q)^{\perp}$}{A lower bound on the minimum weight of CPG(m,q)perp}} \label{sec:lowerbound} 
	
	In this section,  we will show a new lower bound for the minimum weight of a codeword ${\bf s }$ of $\cC_{\PG}(m,q)^{\perp}$ and for the minimum size of $(0 \bmod \,p)$-multisets of $\PG(m,q)$. 
	Firstly, we extend of \cite[Proposition 1]{DeResmini} to multisets.
	
	\begin{lemma}
		\label{lastlemma}
		Let $w_{m-1}$ denote the minimum weight of the code  $\cC^\perp_{\mathrm{PG}}(m-1,q)$ and let $\sigma_{m-1}$ denote the minimum size of a non-empty $(0 \bmod p)$-multiset in $\PG(m-1,q)$, $q=p^h$. 
		If $\cM=(\cS, \mu)$ is a $(0 \bmod p)$-multiset of $\PG(m,q)$ which meets every hyperplane, then
		$$|\cS| \geq qw_{m-1} \quad \textnormal{ and } \quad  |\cM| \geq q\sigma_{m-1}.$$
	\end{lemma}
	\begin{proof}
		Let $\cM=(\cS, \mu)$ be a $(0 \bmod p)$-multiset of $\PG(m,q)$ such that the hyperplanes of $\PG(m,q)$ meet $\cS$ in $0< n_1 < n_2 < \cdots < n_k$ points, and there are $z_{i}$ hyperplanes meeting $\cS$ in $n_i$ points. 
		Clearly 
		\begin{equation}
			\label{last1}
			z_{1} + z_{2} + \cdots + z_k= \frac{q^{m+1} - 1}{q - 1}. 
		\end{equation}
		By double counting the pairs $(x,\mathcal{H})$ where $x\in \cS$ and $\mathcal{H}$ is a hyperplane incident with $x$, we obtain:
		\begin{equation}
			\label{last2}
			n_1 z_{1} + n_2 z_{2} + \cdots + n_kz_k = \vert \cS \vert \frac{q^m - 1}{q - 1}.
		\end{equation}
		Multiplying \eqref{last1} by $n_1$ and subtracting \eqref{last2} we get
		\[
		\vert \cS \vert \geq n_1 \frac{q^{m+1} - 1}{q^m - 1} \geq n_1 q.
		\]
		Since $n_1$ is at least $w_{m-1}$ the first part of the statement is proved.
		To prove the second part, assume that hyperplanes meet $\cM$ in $0<n_1<\ldots<n_k$ points (counting the points with their multiplicities) and that there are $z_i$ hyperplanes meeting $\cM$ in $n_i$ points. Clearly \eqref{last1} holds and by double counting the $(x,\mathcal{H})$ pairs (counting each of them $\mu(x)$ times), we obtain \eqref{last2} with $|\cM|$ on its right side (instead of $|\cS|$). Similarly to the arguments from above, now we obtain  
		$\vert \cM \vert \geq n_1q$ and the result follows noting that $n_1 \geq \sigma_{m-1}$.
	\end{proof}
	
	We will need the following triviality.
	
	\begin{proposition}
		\label{lastproposition}
		The minimum weight $w_1$ of \,$\cC_{\PG}(1,q)^\perp$, $q=p^h$, is $2$ and  
		\[\sigma_1=\min  \left \{\sigma(\mathbf{s}) : \mathbf{s} \in \cC_{\PG}(1,q)^\perp \setminus  \{\boldsymbol{0}\} \right \}=p.\] 
	\end{proposition}
	
	
	In the next proof we will need the following definition.
	
	\begin{definition}
		For a multiset $\cM=(\cS,\mu)$ of $\Sigma$ we define $(p-1)\cM$ as the multiset $(\cS,\bar{\mu})$ of $\Sigma$, where
		for each $x\in \cS$, $\bar{\mu}(x) \in \{1,2,\ldots,p-1\}$ such that $\bar{\mu}(x) \equiv (p-1)\mu(x) \pmod p$. 
	\end{definition}
	
	\begin{theorem}
		\label{codebound2}
		If ${\bf s}\in \cC_{\PG}(m,q)^\perp$ and $q>p$, then
		\[ w({\bf s}) \geq 2(q^{m-1}(p-1)/p+q^{m-2}).\]
	\end{theorem}
	\begin{proof}
		First, take a codeword ${ \bf s}\in \cC_{\AG}(m,q)^\perp$ and let $\cM$ be the multiset associated with ${ \bf s}$.
		Scaling the codeword ${ \bf s}$ does not alter its weight. Thus we may assume the existence of a point  $x \in \cM$ with multiplicity $1$. If there are no points with multiplicity $p-1$, then counting the points on lines through $x$ gives
		$$w({\bf s}) \geq 1+2(q^{m-1}+q^{m-2}+\ldots+q+1),$$
		and hence the assertion follows.
		If there is a point with multiplicity $p-1$, the multiset $(p-1)\cM$ has a point of multiplicity $1$. Since $|\cM|+|(p-1)\cM|$ is $p$ times the weight of $\bf{s}$, by \eqref{affine_bound} (applied twice) we obtain
		\begin{equation}
			\label{AB}
			2((p-1)q^{m-1}+pq^{m-2})\leq |\cM|+|(p-1)\cM|=p\,w({\bf s}),\end{equation}
		from which the result follows. 
		For the origin of the trick applied in \eqref{AB} see {\cite[Result 3.2 and Remark 3.4]{adriaensen}}.
		
		Now, let us assume ${ \bf s } \in \cC_{\PG}(m,q)^{\perp}$ and  let $\cM=(\cS, \mu)$ be the multiset associated with it. If there is a hyperplane of $\PG(m,q)$ disjoint from $\cM$, the result follows from the first part of the proof. Now, suppose that each hyperplane of $\PG(m,q)$ meets $\cM$ and hence $\cS$.
		Take first $m = 2$. Then, by Lemma \ref{lastlemma} and Proposition \ref{lastproposition},
		$$w({\bf s})= \vert \cS \vert \geq 2q \geq 2 (q(p-1)/p +1),$$ as required. Now suppose that $m > 2$ and, by the induction hypothesis, assume that the weight $w_{m-1}$ of each non-zero codeword of $\cC^\perp_{\PG}(m-1,q)$ is at least $2(q^{m-2}(p-1)/p+q^{m-3})$. Thus, by Lemma \ref{lastlemma}
		\[
		w( {\bf s}) = \vert \cS \vert  \geq 
		qw_{m-1}  \geq   2(q^{m-1}(p-1)/p+q^{m-2}),
		\]
		which completes the proof. 
	\end{proof}
	
	Let $\cC^{(k)}_{\PG}(m,q)$  be the $p$-ary code associated with the design $\mathcal{D}^{(k)}_{\PG}(m,q)$ of points and $k$-spaces  of $\PG(m,q)$, i.e. the $\F_p$-subspace generated by the incidence vectors of the blocks of the corresponding design. As said before, the minimum weight of $\mathcal{C}^{(k)}_{\mathrm{PG}}(m,q)^\perp$ is equal to the minimum weight of $\mathcal{C}_{\mathrm{PG}}(m-k+1)^\perp$, see \cite[Theorem 10]{LaStoVan}. Then, as a corollary of Theorem \ref{codebound2}, we get the following result.
	
	\begin{corollary}
		\label{improvesM}
		If $\boldsymbol{s} \in \cC^{(k)}_{\PG}(m,q)$ and $q >p$, then 
		\[ w({\bf s}) \geq 2(q^{m-k}(p-1)/p+q^{m-k-1}).\]  
	\end{corollary}
	We remark that Corollary \ref{improvesM} improves the bounds proved in  \cite[Theorems 14 and 15]{LaStoVan}.
	
	\medskip
	
	In the next propositions we will frequently use $\lambda {\bf s}$, where ${\bf s}=(s_1,\ldots,s_n)$ is a vector with coordinates in $\F_p$ and $\lambda$ is an integer. Of course, the $i$-the coordinate of $\lambda {\bf s}$ is defined to be the unique element of $\{0,1,\ldots,p-1\}$ congruent to $\lambda s_i$ modulo $p$.
	
	\begin{proposition}\label{t/p-t}
		Let ${\bf s} \in \cC_{\PG}(m,q)^\perp$ with $q=p^h$, $h>1$ and $p$ an odd prime power. Let $1 \leq t \leq  \frac{p-1}{2}$. If the non-zero coordinates of $\bf s$ are equal to $t$ or $p-t$, then
		\begin{equation*}
			\sigma({\bf s})\geq (p-1)(q^{m-1}+q^{m-2})+q^{m-2}.
		\end{equation*}
	\end{proposition}
	\begin{proof}
		Firstly, we prove the result for a codeword in ${ \bf s} \in \cC_{\AG}(m,q)^{\perp}$. Let us suppose that ${ \bf s }$ has all non-zero  coordinates equal to $t>1$. Let $\lambda$ be an integer such that $\lambda t \equiv 1 \pmod p$, then $\sigma( { \bf s}) \geq \sigma(\lambda { \bf s} ) \geq (p-1)(q^{m-1}+q^{m-2})+q^{m-2}$ (by Theorem \ref{codebound}). A similar argument applies if all the non-zero coordinates of ${\bf s}$ are equal to $p-t$. If ${\bf s}$ has a coordinate equal to $1$, then the result follows again from Theorem \ref{codebound}. 
		If this is not the case, then denote by $x_i$ the number of coordinates in ${\bf s}$ equal to $i,$ with $i \in \{ t, p-t\}$. Then $\sigma ({\bf s})= t x_t + (p-t) x_{p-t}$. Let $\lambda,\mu$ be integers such that $\lambda t \equiv 1 \pmod p$ and $\mu t \equiv p-1 \pmod p$, respectively. Then we get
		\[\sigma(\lambda{\bf s})=x_t+(p-1)x_{p-t},\]
		\[\sigma(\mu {\bf s})=(p-1)x_t+x_{p-t}.\]
		If $\sigma ({\bf s}) \geq \sigma (\lambda {\bf s})$ or $\sigma ({\bf s}) \geq \sigma (\mu {\bf s})$, the result directly follows from Theorem \ref{codebound}. This will always occurs: indeed, if otherwise $\sigma ({\bf s}) < \sigma (\lambda {\bf s})$ and $\sigma ({\bf s}) < \sigma (\mu {\bf s})$, we have
		\[\sigma( {\bf s }) = tx_t+(p-t)x_{p-t}< x_t + (p-1)x_{p-t},\]
		\[\sigma( {\bf s }) = tx_t+(p-t)x_{p-t}< (p-1)x_t+x_{p-t},\]
		which leads to $x_{t} < x_{p-t} < x_t$, a contradiction.
		
		Now, let us assume ${ \bf s } \in \cC_{\PG}(m,q)^{\perp}$ and  let $\cM$ be the multiset associated with it. If there is a hyperplane of $\PG(m,q)$ disjoint from $\cM$, the result follows from the first part of the proof. Now, suppose that each hyperplane of $\PG(m,q)$ meet $\cM$.
		Take first $m = 2$, then  by Lemma \ref{lastlemma} and Proposition \ref{lastproposition},
		$$\sigma({\bf s})= \vert \cM \vert \geq p q \geq (p-1)(q+1)+1,$$ as required. Now suppose that $m > 2$ and, by induction hypothesis,  we may assume that for any non-zero codeword in $\cC^\perp_{\PG}(m-1,q)$ the sum of its coordinates is at least $(p-1)(q^{m-2}+q^{m-3})+q^{m-3}$.  Thus, by Lemma \ref{lastlemma}
		\[
		\sigma( {\bf s}) = \vert \cM \vert  \geq 
		q \sigma_{m-1}  \geq (p-1)(q^{m-1}+q^{m-2})+q^{m-2}  ,
		\]
		which completes the proof. 
	\end{proof}
	
	As a direct consequence of this result, we get the following.
	
	\begin{corollary}
		If ${\bf s}\in \cC_{\PG}(m,q)^\perp$ with  $q=3^h$ and $h>1$, then
		\begin{equation}
			\sigma({\bf s}) \geq 2q^{m-1}+3q^{m-2}.
		\end{equation}

	\end{corollary}
	
	We conclude the section by showing the following achievement in which we slightly relax the hypothesis of Proposition \ref{t/p-t}, providing $p=5$. 
	
	
	
	%
	
	\begin{proposition}
		\label{codebound3}
		If ${\bf s}\in \cC_{\PG}(m,q)^\perp$ with  $q=5^h$ and $h>1$, then 
		\[\sigma({\bf s}) \geq 4q^{m-1}+5q^{m-2}.\]
	\end{proposition}
	\begin{proof}
		
		First consider the case ${ \bf s } \in \cC_{\AG}(m,q)^\perp$ and assume that ${\bf s}$ has a non-zero coordinate which is equal to $1$ or $4$.  If  ${\bf s}$ has a coordinate equal to $1$, then the result follows from Theorem \ref{codebound}. If this is not the case, then denote by $x_i$ the number of $i$ coordinates in ${\bf s}$ for $i=2,3,4$. Then 
		\[\sigma({\bf s})=2x_2+3x_3+4x_4,\]
		\[\sigma(2{\bf s})=4x_2+x_3+3x_4,\]
		\[\sigma(3{\bf s})=x_2+4x_3+2x_4,\]
		\[\sigma(4{\bf s})=3x_2+2x_3+x_4.\] 
		If $x_3\neq 0$ and  $\sigma(2{\bf s})$ is less than or equal to $\sigma({\bf s})$, then the assertion follows from Theorem \ref{codebound}. 
		The same holds when $x_2\neq 0$ and $\sigma(3{\bf s})$ is less than or equal to $\sigma({\bf s})$. Since $x_4\neq 0$, if $\sigma(4{\bf s})$ is less than or equal to $\sigma({\bf s})$, then we are done by Theorem \ref{codebound}. From now on, we assume the contrary.
		
		If $x_2=x_3=0$, then  $\sigma( {\bf s}) \geq \sigma (4{\bf s})$, contradicting 
		$\sigma(4 {\bf s})>\sigma({\bf s})$. If $x_2\neq 0$ and $\sigma(3{\bf s}) > \sigma({\bf s})$, then
		
		\[\sigma(3{\bf s})=x_2+4x_3+2x_4 > 2x_2+3x_3+4x_4,\]
		\[\sigma(4{\bf s})=3x_2+2x_3+x_4 > 2x_2+3x_3+4x_4,\]
		and summing these two inequalities yields a contradiction. It follows that $x_2=0$.
		
		If $x_2=0$, $x_3\neq 0$, then we may assume 
		\[\sigma(2{\bf s})=x_3+3x_4 > 3x_3+4x_4,\]
		\[\sigma(4{\bf s})=2x_3+x_4 > 3x_3+4x_4.\]
		Similarly, summing up these two inequalities gives a contradiction. 
		Then, for any ${ \bf s } \in \cC_{\AG}(m,q)^\perp$ with a non-zero coordinate equal to $1$ or $4$, $\sigma({\bf s}) \geq 4 q^{m-1}+5q^{m-2}$. 
		
		By applying induction as in Proposition \ref{t/p-t}, we obtain that the result holds for  codewords in $\cC_{\PG}(m,q)^\perp$ with a non-zero coordinate equal to $1$ or $4$.
		Now, let us suppose that the non-zero coordinates of ${ \bf s}$ are equal to $2$ or $3$. Then the result follows directly from Proposition \ref{t/p-t}.
	\end{proof}
	
	\section{Acknowledgment}
	This work was supported by the Italian National Group for Algebraic and Geometric Structures and their Applications (GNSAGA–INdAM). The first author gratefully acknowledges the hospitality of the University of Naples Federico II. The second author would like to express his gratitude for the hospitality received at E\"otv\"os Lor\'and University (ELTE), Budapest.

	\vspace{1cm}
	
	\noindent Giovanni Longobardi, Giuseppe Marino, Rocco Trombetti,\\
	Dipartimento di Matematica e Applicazioni “Renato Caccioppoli”\\
	Università degli Studi di Napoli Federico II,\\
	via Cintia, Monte S. Angelo I-80126 Napoli, Italy. \\
	email: \\
	\texttt{\{giovanni.longobardi, giuseppe.marino rtrombet\}@unina.it}
	
	\vspace{1.5cm}
	
	\noindent Bence Csajb\'{o}k,\\
	Department of Computer Science\\
	ELTE E\"otv\"os Lor\'and University\\
	H-1117 Budapest, P\'azm\'any P.\ stny.\ 1/C, Hungary.\\ E-mail: {\texttt{bence.csajbok@ttk.elte.hu}}
	


\begin{thebibliography}{pippo}
		
		
		
		\bibitem{adriaensen}
		\textsc{S. Adriaensen:} A note on small weight codewords of projective geometric codes and on the smallest stes of even type, SIAM J.\ Discrete Math.\ 37(3)  (2023), 2072--2087.
		
		\bibitem{AssmusKey}
		\textsc{E. F. Assmus, Jr., J. D. Key:} Designs and their Codes, Cambridge Tracts in Mathematics, Cambridge
		University Press, 103 (1992) (Second printing with corrections, 1993).
		
		
		\bibitem{BagchiInamdar}
		\textsc{B. Bagchi, S. P. Inamdar:} Projective geometric codes, J.\ Combin.\ Theory Ser.\ A, 99(1) (2002), 128--142.
		
		\bibitem{5nomi}
		{\sc S. Ball, A. Blokhuis, A. G\'acs, P. Sziklai, Zs. Weiner:} On linear codes whose weights and length have a common divisor, Adv.\ Math.\ 211 (2007), 94--104.
		
		
		
		\bibitem{evasive}
		\textsc{D. Bartoli, B. Csajbók, G. Marino, R. Trombetti:} Evasive subspaces, J.\ Comb.\ Des.\ 29(8) (2021), 533--551. 
		
		\bibitem{scatteredsubspaces}
		\textsc{A. Blokhuis, M. Lavrauw:} Scattered Spaces with Respect to a Spread in $\PG(n,q)$, Geom.\ Dedicata 81 (2000), 231--243.
		
		\bibitem{Bonoli-Polverino}
		{\sc G. Bonoli, O. Polverino:} $\F_q$-linear blocking sets in $\PG(2, q^4)$, Innov.\  Incidence Geom.\ 2 (2005), 35--56.
		
		
		\bibitem{DeResmini}
		{\sc N.J. Calkin, J. D. Key, M.J. De Resmini:} Minimum Weight and Dimension Formulas for
		Some Geometric Codes, Des.\ Codes Cryptogr.\ 17 (1999), 105--120.
		
		
		
		
		\bibitem{deboeckvan}
		\textsc{M. De Boeck, G. Van de Voorde:} Embedded antipodal planes and the minimum weight of the dual code of points and lines in projective planes of order $p^2$, Des.\ Codes Cryptogr.\ (2022), 1--26.
		
		
		\bibitem{Delsarte}
		\textsc{P. Delsarte:} BCH bounds for a class of cyclic codes, SIAM J.\ Appl.\ Math.\ 19 (1970), 420--429.
		
		
		
		
		\bibitem{LaStoVan}
		\textsc{M. Lavrauw, L. Storme, G. Van de Voorde:} On the code generated by the incidence matrix of points and $k$-spaces in $\PG(n,q)$ and its dual, Finite Fields Appl.\ 14(4) (2008), 1020--1038.
		
		
		\bibitem{linearcodes}
		\textsc{M. Lavrauw, L. Storme, G. Van de Voorde:} Linear codes from projective spaces, Error-Correcting Codes, Finite Geometries and Cryptography
		Edited by: Aiden A. Bruen and
		David L. Wehlau, Contemporary Mathematics 523 (2010).
		
		\bibitem{survey}
		\textsc{G. Longobardi:} Scattered polynomials: an overview on their properties, connections and applications, Art Discrete Appl.\ Math.\ (2025) DOI: \url{https://doi.org/10.26493/2590-9770.1842.2f0}
		
		
		\bibitem{payne_complete_1971}
		\textsc{S.~E. Payne:}
		A complete determination of translation ovoids in finite Desarguian planes, Atti Accad.\ Naz.\ Lincei Rend.\ Cl.\ Sci.\ Fis.\ Mat.\ Nat.\  51 (1971), 328--331.
		
		\bibitem{linearset}
		{\sc O. Polverino:} Linear sets in finite projective spaces, Discrete Math.\ 310 (2010), 3096--3107.
		
		
		
		
		
		
	\end{thebibliography}
\end{document}